\documentclass[11pt]{amsart}
\usepackage{graphicx}
\usepackage{verbatim}
\usepackage{textcomp}
\usepackage{amssymb}
\usepackage{cite}
\usepackage{amsmath}
\usepackage{latexsym}
\usepackage{amscd}
\usepackage{amsthm}
\usepackage{mathrsfs}
\usepackage{xypic}
\usepackage{bm}
\usepackage{hyperref}
\usepackage{url}

\newtheorem{thm}{Theorem}[section]
\newtheorem{corr}[thm]{Corollary}
\newtheorem{lem}[thm]{Lemma}
\newtheorem{prop}[thm]{Proposition}

\theoremstyle{definition}
\newtheorem{defn}{Definition}[section]

\theoremstyle{remark}
\newtheorem{rem}{Remark}[section]
\numberwithin{equation}{section}
\def\R{\mathbb R}

\def\SS{\mathbb S}
\def\cal{\mathcal}
\def\f{\frac}
\def\td{\tilde}
\def\ra{\rightarrow}
\def\pt{\partial}

\begin{document}
\title[$\mathcal{F}$-stability for self-shrinkers]{$\mathcal{F}$-stability for self-shrinking solutions to mean curvature flow}
\author{Ben Andrews}
\address{Mathematical Sciences Institute, Australia National University; Mathematical Sciences
Center, Tsinghua University; and Morningside Center for Mathematics, Chinese Academy of Sciences.}
\email{Ben.Andrews@anu.edu.au}
\author{Haizhong Li}
\address{Department of mathematical sciences, Tsinghua University, 100084, Beijing, P. R. China}
\email{hli@math.tsinghua.edu.cn}
\author{Yong Wei}
\address{Department of mathematical sciences, Tsinghua University, 100084, Beijing, P. R. China}
\email{wei-y09@mails.tsinghua.edu.cn}
\thanks{The research of the first author was partially supported by Discovery Projects grant DP120100097 of the Australian Research Council.
The Research of the second author was supported by NSFC No. 10971110.}
\maketitle

\begin{abstract}
In this paper, we formulate the notion of the $\mathcal{F}$-stability of self-shrinking solutions to mean curvature flow in arbitrary codimension. Then we give some classifications of the $\mathcal{F}$-stable self-shrinkers in arbitrary codimension, in codimension one case, our results reduce to Colding-Minicozzi's results.
\end{abstract}

\section{Introduction}

Self-shrinking solutions often arise as the tangent flow at singularities of the mean curvature flow, which is one of the most important part in the study of mean curvature flow. We call any time-slice of the self-shrinking solutions a self-shrinker, which is an $n$-dimensional submanifold in the Euclidean space $\R^{n+p}$ satisfying
\begin{equation}
    H=-\f{(x-x_0)^{\perp}}{2t_0},\qquad x_0\in\R^{n+p}, 0<t_0\in\R.
\end{equation}

There are many interesting works about the classification of self-shrinkers in recent years. Most recently, in the paper \cite{CM2}, Colding and Minicozzi studied the entropy stability of self-shrinkers in the mean curvature flow (in the case $p=1$). Given $x_0\in\R^{n+p}$ and $t_0>0$, the functional $\mathcal{F}_{x_0,t_0}$ is defined by
\begin{eqnarray}
    \mathcal{F}_{x_0,t_0}(M)&=&(4\pi t_0)^{-\f n2}\int_Me^{-\f{|x-x_0|^2}{4t_0}}d\mu,
\end{eqnarray}
which can be traced back to Huisken's Monotonicity formula \cite{H2}. The critical points of $\mathcal{F}_{x_0,t_0}$ are the surfaces shrinking to $x_0$ at time $t_0$. The entropy functional $\lambda=\lambda(M)$ is the supremum of the $\mathcal{F}$ functional over $x_0$ and $t_0$:
\begin{eqnarray}
\lambda(M) &=& \sup_{x_0\in\R^{n+p},t_0>0}\mathcal{F}_{x_0,t_0}(M).
\end{eqnarray}
The entropy is non-increasing in $t$ along the mean curvature flow, and its critical points are precisely given by self-shrinkers.

A self-shrinker is called entropy-stable if it is a local minimum for the entropy functional. One is more interested in the stable self-shrinkers since the unstable ones could be perturbed away thus may not represent generic singularities.  Colding and Minicozzi \cite{CM2} showed that the shrinking spheres, cylinders and planes are the only entropy-stable self-shrinkers under the mean curvature flow. Their proof involves three main steps:
\begin{itemize}
  \item Show that entropy-stable self-shrinkers that do not split off a line must be $\mathcal{F}$-stable;
  \item Show that $\mathcal{F}$-stability implies mean convexity (i.e. $H\geq 0$); and
  \item Classify the mean convex self-shrinkers.
\end{itemize}
See the definition of $\mathcal{F}$-stable in section 5.

In higher codimension case, the key step in the classification of entropy stable self-shrinkers is also classifying the $\mathcal{F}$-stable self-shrinkers. In this paper, we first formulate the notion of the $\mathcal{F}$-stability for self-shrinkers in arbitrary codimension. In section 3 and 4, we will calculate the full first and second variation formula of the $\mathcal{F}$-functional. The $\mathcal{F}$-stability of a self-shrinker is related to the eigenvalues of a bilinear symmetric form $\mathcal{I}$ on the space of cross-sections in $NM$. In section 5, we will give an necessary condition for closed $\mathcal{F}$-stable self-shrinkers, that is $\{-2,-1\}$ are the only negative eigenvalues of the bilinear symmetric form $\mathcal{I}$.

In section 6 and 7, we will classify the $\mathcal{F}$-stable self-shrinkers in higher codimension. As the codimension increases, the situation becomes more complicated. In section 6, we prove that the sphere $\SS^n(\sqrt{n})$ in $\R^{n+p}$ is the only $\mathcal{F}$-stable one of the minimal submanifolds in spheres, and our proof is inspired in part by the work of J. Simons \cite[Theorem 5.1.1]{Simon} on the instability of minimal submanifolds of spheres.

\begin{thm}
Let $M^n$ be a closed minimal submanifold of $\SS^{n+p-1}(\sqrt{n})\subset\R^{n+p}$. If $M$ is $\mathcal{F}$-stable, then $M$ is the $n$-sphere $\SS^n(\sqrt{n})$.
\end{thm}

Then we consider the special case ``self-shrinkers with parallel principal normal", i.e, $\nabla^{\perp}\nu=0$, where $\nu$ is a unit length normal vector field parallel to the mean curvature vector field $H$ . While the parallel principal normal condition may seem artificial and is not preserved under the mean curvature flow, it includes the important examples of minimal submanifolds of the sphere $S^{n+p-1}$. We will show that any $\mathcal{F}$-stable self-shrinker with parallel principal normal must be the sphere $\SS^n(\sqrt{n})$ or the plane $\R^n$. We note that in codimension one case, our results reduce to Colding-Minicozzi's results in [6].

\begin{thm}\label{thm-1-2}
Let $x:M^n\ra\R^{n+p}$ be a closed $\mathcal{F}$-stable self-shrinker with parallel principal normal, then $M^n$ is the sphere $\SS^n(\sqrt{n})$.
\end{thm}
\begin{thm}\label{thm-1-3}
Let $x:M^n\ra\R^{n+p}$ be a complete noncompact embedded $\mathcal{F}$-stable self-shrinker with parallel principal normal, with polynomial volume growth and without boundary. Suppose further $|A|^2-|A^{\nu}|^2\leq c$ for some constant $c$, where $A^{\nu}=<\nu,A>$ is the principal second fundamental form. Then $M$ must be the plane $\R^n$.
\end{thm}

In the last section, we relate the entropy-stability to the $\mathcal{F}$-stability for self-shrinkers in arbitrary codimension. We show that Colding-Minicozzi's result that an entropy-stable self-shrinker which does not split off a line is $F$-stable also holds true for higher codimension.
Then as Corollarys of Theorem \ref{thm-1-2} and \ref{thm-1-3}, we obtain two classification results of the entropy-stable self-shrinkers.

\section{$\mathcal{F}$-functional and self-shrinking solutions to mean curvature flow}

In this section, we recall the $\mathcal{F}$-functional, Huisken's monotonicity formula, self-shrinking solutions and singularities of the mean curvature flow. Most of the material here is stated in Huisken and Colding-Minicozzi II's paper for the hypersurface case, but it also holds for arbitrary codimension without any change in the proof.

Let $x:M^n\times [0,T)\ra\R^{n+p}$ be a one-parameter family of submanifolds in $\R^{n+p}$, which is called a mean curvature flow if the position vector $x$ evolves in the direction the mean curvature vector $H$, i.e.,
\begin{equation}\label{mcf}
    \f {\pt x}{\pt t}=H.
\end{equation}
Let $\Phi_{(x_0,t_0)}(x,t)=(4\pi(t_0-t))^{-\f n2}\exp (-\f{|x-x_0|^2}{t_0-t})$ be the backward heat kernel. Huisken proved the following monotonicity formula for the mean curvature flow \cite{H2}:
\begin{align}\label{monoto}
    \f d{dt}\int_{M_t}\Phi_{(x_0,t_0)}(x,t)d\mu_t=&-\int_{M_t}\biggl|H+\f{(x-x_0)^{\perp}}{2(t_0-t)}\biggr|^2\Phi_{(x_0,t_0)}(x,t)d\mu_t
\end{align}
A solution of \eqref{mcf} is called self-shrinking about $(x_0,t_0)$ if it satisfies
\begin{align}\label{self-shrink}
    H=-\f{(x-x_0)^{\perp}}{2(t_0-t)}.
\end{align}
A submanifold is said to be a self-shrinker if it is the time $t=0$ slice of a solution which is self-shrinking about $(x_0,t_0)$. That is, we call the submanifold $x:M^n\ra \R^{n+p}$ satisfying
\begin{align}\label{shrinker}
    H=-\f{(x-x_0)^{\perp}}{2t_0}.
\end{align}
a self-shrinker.  The $\mathcal{F}$-functional is defined as
\begin{align*}
    \mathcal{F}_{x_0,t_0}(M_t)=&(4\pi t_0)^{-\f n2}\int_{M_t}e^{-\f {|x-x_0|^2}{4t_0}}d\mu_t\\
    =&\int_{M_t}\Phi_{(x_0,t_0)}(x,0)d\mu_t\\
    =&\int_{M_t}\Phi_{(x_0,t_0+t)}(x,t)d\mu_t.
\end{align*}
It is easy to see that $\mathcal{F}$-functional is invariant under scalings, i.e., for any $\alpha>0$, we have
\begin{equation}
    \mathcal{F}_{\alpha x_0, \alpha^2t_0}(\alpha M_t)=\mathcal{F}_{x_0,t_0}(M_t).
\end{equation}
Moreover, Huisken's monotonicity formula \eqref{monoto} implies for $t>s$,
\begin{align*}
    \mathcal{F}_{x_0,t_0}(M_t)=&\int_{M_t}\Phi_{(x_0,t_0+t)}(x,t)d\mu_t\\
    \leq &\int_{M_s}\Phi_{(x_0,t_0+t)}(x,s)d\mu_s\\
    =&\int_{M_s}\Phi_{(x_0,t_0+t-s)}(x,0)d\mu_s\\
    =&\mathcal{F}_{x_0,t_0+t-s}(M_s).
\end{align*}
The entropy of a submanifold is defined as the supremum of the $\mathcal{F}$-functional
\begin{align}\label{entropy}
    \lambda(M_t)=\sup_{x_0\in\R^{n+p},t_0>0}\mathcal{F}_{x_0,t_0}(M_t).
\end{align}
Then for any small positive constant $\epsilon$, there exists a pair $(x_0,t_0)$ such that
\begin{align*}
     \lambda(M_t)-\epsilon\leq \mathcal{F}_{x_0,t_0}(M_t),
\end{align*}
and therefore,
\begin{align*}
    \lambda(M_t)-\epsilon\leq \mathcal{F}_{x_0,t_0}(M_t)\leq \mathcal{F}_{x_0,t_0+t-s}(M_s)\leq \lambda(M_s).
\end{align*}
Since $\epsilon$ is arbitrary, we have
\begin{align}
    \lambda(M_t)\leq \lambda(M_s),\qquad (t>s).
\end{align}
We summarize the result in the following proposition.
\begin{prop}[Lemma 1.11 in \cite{CM2}]
The entropy $\lambda(M_t)$ defined in \eqref{entropy} is non-increasing in $t$ under the mean curvature flow.
\end{prop}

Self-shrinkers play an important role in the study of mean curvature flow, since they describe all the possible tangent flows at a given singularity of a mean curvature flow.   Huisken first used the monotonicity formula to prove that the flow is asymptotically self-similar near type-I singularities. Later Ilmanen and White extended Huisken's formula to weak solutions and proved asymptotic self-similarity for tangent flows at all singularities, in the following sense:

At a given singularity $(x_0,t_0)$ of the mean curvature flow, we can blow up by setting
\begin{align*}
    M_t^j=c_j\biggl(M_{c_j^{-2}t+t_0}-x_0\biggr).
\end{align*}
That is, by first translating $M_t$ in space-time to move $(x_0,t_0)$ to $(0,0)$ and then taking a sequence of parabolic dilations $(x,t)\ra (c_jx,c_j^2t)$ with $c_j\ra\infty$. By using Huisken's monotonicity formula and the standard compactness theorem, we can extract a subsequence of $M_t^j$'s converging weakly to a limiting flow, which will be called a \emph{tangent flow} at $(x_0,t_0)$. A tangent flow will achieve equality in the monotonicity formula, and will be a self-shrinking solution to the mean curvature flow.

By using Huisken's monotonicity formula, Colding-Minicozzi proved that
\begin{prop}
Any time-slice of a tangent flow at a given singularity has polynomial volume growth.
\end{prop}

We say that a submanifold $M^n$ in $\R^{n+p}$  has polynomial volume growth if there exist constants $C$ and $d$ such that for all $r\geq 1$, there holds
\begin{equation*}
    \textrm{Vol} (B(r)\cap M)\leq Cr^d,
\end{equation*}
where $B_r$ denotes an Euclidean ball with radius $r$. The proposition says that any self-shrinker which arises as the blow up at a given singularity in the mean curvature flow must have polynomial volume growth.

\section{The first variation of $\mathcal{F}$-functional}

In this section, we will give a variational characterization of self-shrinkers. Let $M^n$ be a submanifold in $\R^{n+p}$, we choose an orthonormal frame field $\{e_i,e_{\alpha}\}$ along $M$ such that $e_i\in TM$ and $e_{\alpha}\in NM$. Throughout the paper, we always make the following convention on the range of the indices:
\begin{equation*}
    1\leq i,j,k\leq n;\qquad n+1\leq \alpha,\beta,\gamma\leq n+p
\end{equation*}
Denote $h^{\alpha}_{ij}$ the components of the second fundamental form of $M$, $\sigma_{\alpha\beta}=\sum\limits_{i,j}h^{\alpha}_{ij}h^{\beta}_{ij}$.

We say $M_s\subset\R^{n+p}$ is a variation of $M$ if it is given by a one parameter family of immersions $X_s:M\ra\R^{n+p}$ with $X_0$ equal to the identity. The vector field $\f{\pt}{\pt s}\mid_{s=0}X_s=V$ is called the variation vector field. We only consider the normal variation vector field $V$, which can be expressed as $V=\sum\limits_{\alpha}V^{\alpha}e_{\alpha}$.

For the functional $\mathcal{F}_{x_0,t_0}$ defined in the introduction, we say that $M$ is a critical point for $\mathcal{F}_{x_0,t_0}$ if it is critical with respect to all normal variations in $M$ and all variations in $x_0$ and $t_0$. Let $x_s$ and $t_s$ be variations of $x_0$ and $t_0$ with $x_0'=y$ and $t_0'=h$.
Recall that the volume element of  $M_s\subset\R^{n+p}$ is given by
\begin{equation*}
    d\mu_s=\sqrt{\det g(x,s)}dx=\f{\sqrt{\det g(x,s)}}{\sqrt{\det g(x,0)}}d\mu_0.
\end{equation*}
We define the function
\begin{equation*}
    J(x,s)=\f{\sqrt{\det g(x,s)}}{\sqrt{\det g(x,0)}}
\end{equation*}
Then we have (cf. \cite{L})
\begin{align}
    J'(x,0)=&-<V,H>,\label{j-1}\\
    J''(x,0)=&|\nabla^{\perp}V|^2-\sum_{\alpha,\beta}\sigma_{\alpha\beta}V^{\alpha}V^{\beta}+<V,H>^2\nonumber\\
             &\qquad +\textrm{div} (\bar{\nabla}_VV)^T-<\bar{\nabla}_VV,H>,\label{j-2}
\end{align}
where $\nabla^{\perp}$, $\textrm{div}$ denote the normal connection of the normal bundle and the divergence on $M$, respectively.

Next, we denote the integrand of the functional $\mathcal{F}_{x_s,t_s}(M_s)$ by
\begin{equation*}
    I(s)=(4\pi t_s)^{-\f n2}e^{-\f{|x-x_s|^2}{4t_s}},
\end{equation*}
then by a direct calculation, we get
\begin{align}
    I'(0)=&I(0)\left(-\left\langle\f{x-x_0}{2t_0},V-y\right\rangle+h\left(\f{|x-x_0|^2}{4t_0^2}-\f n{2t_0}\right)\right)\label{I-1}\\
    I''(0)=&I(0)\left[\left(-<\f{x-x_0}{2t_0},V-y>+h(\f{|x-x_0|^2}{4t_0^2}-\f n{2t_0})\right)^2\right.\nonumber\\
    &\qquad \left.-<(\f{x-x_0}{2t_0})',V-y>-<\f{x-x_0}{2t_0},(V-y)'>\right.\label{I-2}\\
    &\qquad \left.h'(\f{|x-x_0|^2}{4t_0^2}-\f n{2t_0})+h(\f{|x-x_0|^2}{4t_0^2}-\f n{2t_0})'\right]\nonumber
\end{align}

From \eqref{j-1} and \eqref{I-1}, we obtain the following first variation formula:
\begin{lem}\label{lem-v1}
Let $M_s\subset\R^{n+p}$ be a variation of $M$ with normal variation vector field $M_0'=V$. If $x_s$ and $t_s$ are variations of $x_0$ and $t_0$ with $x_0'=y$ and $t_0'=h$, then the first variation of $F_{x_0,t_0}$ is equal to
\begin{eqnarray}\label{v-01}
 \mathcal{F}_{x_0,t_0}'&=& (4\pi t_0)^{-\f n2}\int_M\left[-\left<H+\f {(x-x_0)^{\perp}}{2t_0},V\right>\right.\nonumber\\
&&\qquad +\left.h\left(\f{|x-x_0|^2}{4t_0^2}-\f n{2t_0}\right)+\f {<x-x_0,y>}{2t_0}\right]e^{-\f{|x-x_0|^2}{4t_0}}d\mu\label{v-1}
\end{eqnarray}
\end{lem}
\begin{proof}
From \eqref{j-1} and \eqref{I-1}, we have
\begin{align*}
    \mathcal{F}_{x_0,t_0}'=&\int_M(I'J+IJ')d\mu_0\\
    =&(4\pi t_0)^{-\f n2}\int_M\left[-\left<H+\f {(x-x_0)^{\perp}}{2t_0},V\right>\right.\\
&\qquad +\left.h\left(\f{|x-x_0|^2}{4t_0^2}-\f n{2t_0}\right)+\f {<x-x_0,y>}{2t_0}\right]e^{-\f{|x-x_0|^2}{4t_0}}d\mu_0
\end{align*}
\end{proof}

Equations \eqref{j-2} and \eqref{I-2} will be used in the next section to calculate the second variation.

From the first variation formula \eqref{v-01}, we can see that if $M$ is a critical point of $\mathcal{F}_{x_0,t_0}$, then it must satisfy
\begin{equation*}
    H=-\f{(x-x_0)^{\perp}}{2t_0}.
\end{equation*}
We will show that the converse is also true, that is a self-shrinker must be the critical point of the functional $\mathcal{F}_{x_0,t_0}$. For simplicity, we only consider $x_0=0$ and $t_0=\f 12$, i.e., the self-shrinker equation is
\begin{equation}\label{shrinker-1}
    H=-x^{\perp}.
\end{equation}

In the study of self-shrinker, the following elliptic operator was first introduced by Colding-Minicozzi\footnote{Note that our notation is different with Colding-Minicozzi's, that is due to the different normalization of the self-shrinker equation \eqref{shrinker-1}, which is not essential.} (see (3.7) in \cite{CM2}):
\begin{equation}\label{ope-1}
\cal{L}=\Delta-<x,\nabla(\cdot)>=e^{\f{|x|^2}{2}}\textrm{div}(e^{-\f{|x|^2}{2}}\nabla\cdot),
\end{equation}
where $\Delta$, $\nabla$ and $\textrm{div}$ denote the Laplacian, gradient and divergent operator on the self-shrinker respectively, $<\cdot,\cdot>$ denotes the standard inner product in $\R^{n+p}$. We apply the operator $\mathcal{L}$ on some natural geometric quantities and obtain the following basic equations. (cf. \cite{CL2,CM2,LW})
\begin{lem}\label{lem-3-2}
Let $x:M^n\ra \R^{n+p}$ be an n-dimensional complete self-shrinker, then
\begin{align}
    \mathcal{L}x_i=&-x_i,\\
    \f 12\mathcal{L}|x|^2=&n-|x|^2,
\end{align}
where $x_i$ is the i-th component of the position vector $x$.
\end{lem}

The operator $\cal{L}$ is self-adjoint in a weighted $L^2$ space. The next two results were proved by Colding-Minicozzi \cite{CM2} for hypersurface self-shrinkers but can be stated in the same way for self-shrinkers in arbitrary codimension.
\begin{lem}\label{lem2-1}
If $M^n\subset \R^{n+p}$ is a submanifold, $u$ is a $C^1$ function
with compact support, and $v$ is a $C^2$ function, then
\begin{equation}
\int_Mu(\cal{L}v)e^{-\f{|x|^2}{2}}=-\int_M<\nabla v,\nabla
u>e^{-\f{|x|^2}{2}}.
\end{equation}
\end{lem}
\begin{corr}\label{cor2-1}
Suppose that $M^n\subset\R^{n+p}$ is a complete submanifold without
boundary, if $u, v$ are $C^2$ functions with
\begin{equation*}
\int_M(|u\nabla v|+|\nabla u||\nabla
v|+|u\cal{L}v|)e^{-\f{|x|^2}{2}}<+\infty,
\end{equation*}
then we get
\begin{equation}
\int_Mu(\cal{L}v)e^{-\f{|x|^2}{2}}=-\int_M<\nabla v,\nabla
u>e^{-\f{|x|^2}{2}}.
\end{equation}
\end{corr}
We apply the self-adjointness of $\mathcal{L}$ to obtain the following two results, the proof is simple and similar as in Colding-Minicozzi's paper (see Lemma 3.25 and Corollar 3.34 in \cite{CM2}).
\begin{lem}\label{lem-v3-5}
Let $x:M^n\ra \R^{n+p}$ be an n-dimensional complete self-shrinker with polynomial volume growth, and $\omega\in\R^{n+p}$ is a constant vector, then
\begin{align}
    &\int_M\left(|x|^2-n\right)e^{-\f {|x|^2}2}=0,\label{v-2}\\
    &\int_Mxe^{-\f {|x|^2}2}=0=\int_Mx|x|^2e^{-\f {|x|^2}2},\label{v-3}\\
    &\int_M\left(|x|^4-n(n+2)+2|H|^2\right)e^{-\f {|x|^2}2}=0.\label{v-5}\\
    &\int_M<x,\omega>^2e^{-\f {|x|^2}2}=\int_M|w^T|^2e^{-\f {|x|^2}2}.\label{v-6}
\end{align}
\end{lem}

\begin{corr}\label{coor-v1}
Let $M$ be as in Lemma \ref{lem-v3-5}, then
\begin{eqnarray}
\int_M\left[(|x|^2- n)^2- 2n\right]e^{-\f {|x|^2}2}&=& -2\int_M|H|^2e^{-\f {|x|^2}2}.
\end{eqnarray}
\end{corr}

Now we are in the position to show that $(M,x_0,t_0)$ is the critical point of the $\mathcal{F}$-functional is equivalent to that $M$ is the critical point of $\mathcal{F}$ with fixed $x_0$ and $t_0$.

\begin{thm}\label{thm-1st}
$M$ is a critical point for $\mathcal{F}_{x_0,t_0}$ if and only if $H=-\f{(x-x_0)^{\perp}}{2t_0}$.
\end{thm}
\begin{proof}
Without loss of generality, we only show this for $x_0=0$ and $t_0=\f 12$. When $x_0=0$ and $t_0=\f 12$, the first variation formula \eqref{v-1} becomes
\begin{align}
    \mathcal{F}'_{0,\f 12}=&(2\pi)^{-\f n2}\int_M\biggl[-<H+x^{\perp},V>\nonumber\\
    &\qquad+h(|x|^2-n)+<x,y>\biggr]e^{-\f {|x|^2}2}d\mu.\label{v-4}
\end{align}
If $M$ is a critical point for $\mathcal{F}_{0,\f 12}$, then it is obvious that $M$ satisfies $H=-x^{\perp}$. Conversely, if $M$ satisfies $H=-x^{\perp}$, then equations \eqref{v-2} and \eqref{v-3} imply the last two terms in \eqref{v-4} vanish for every $h$ and every $y$. Therefore $M$ is a critical point of $\mathcal{F}_{0,\f 12}$.
\end{proof}

The equations \eqref{v-5}, \eqref{v-6} and the Corollary \ref{coor-v1} will be used in the next section when we compute the second variation of the functional $\mathcal{F}_{x_0,t_0}$ at a critical point.

\section{The second variation of $\mathcal{F}$-functional}

In this section, we calculate the second variation formula for the functional $\mathcal{F}_{x_0,t_0}$ when $M$ is a critical point.
\begin{thm}
Suppose $M$ is a critical point of the functional $\mathcal{F}_{x_0,t_0}$, and $M$ is complete, with polynomial volume growth. If $M_s\subset\R^{n+p}$ is a normal variation of $M$, $x_s$,$t_s$ are variations of $x_0$ and $t_0$ with
 \begin{equation*}
    M_0'=V,\qquad x_0'=y,\qquad t_0'=h,
 \end{equation*}
then for $x_0=0$ and $t_0=\f 12$, we have the second variation formula
\begin{align}
    \mathcal{F}''=&(2\pi)^{-\f n2}\int_M\biggl[|\nabla^{\perp}V|^2-\sum_{\alpha,\beta}\sigma_{\alpha\beta}V^{\alpha}V^{\beta}-|V|^2\nonumber\\
    &\qquad -2h^2|H|^2-4h<H,V>-2<y,V>-|y^{\perp}|^2\biggr]e^{-\f{|x|^2}2}d\mu.\label{vv-3}
\end{align}
\end{thm}
\begin{rem}
For normal vector field $V=\sum\limits_{\alpha}V^{\alpha}e_{\alpha}$, if we define the operator $L$ by
\begin{equation}\label{ope-vv}
    LV^{\alpha}=\Delta^{\perp}V^{\alpha}-<x,e_k>V_{,k}^{\alpha}+\sum_{\alpha\beta}\sigma_{\alpha\beta}V^{\beta}+V^{\alpha},
\end{equation}
where $\Delta^{\perp}V^{\alpha}$ and $V_{,k}^{\alpha}$ denote the component of the Laplacian and first order covariant derivative of the cross-section $V$ of the normal bundle $NM$ (cf. \cite{CL1,CL2,Li,LW}). Then the second variational formula \eqref{vv-3} can be rewritten as
\begin{align}
    \mathcal{F}''=&(2\pi)^{-\f n2}\int_M\biggl[-\sum_{\alpha}V^{\alpha}LV^{\alpha}-2h^2|H|^2\nonumber\\
    &\qquad -4h<H,V>-2<y,V>-|y^{\perp}|^2\biggr]e^{-\f{|x|^2}2}d\mu.\tag{\ref{vv-3}$'$}\label{vv-3'}
\end{align}

\end{rem}

\begin{proof}[Proof of Theorem 4.1]
Since $M$ is a critical point of $\mathcal{F}_{x_0,t_0}$, it follows from the first variation formula that
\begin{align*}
    &H+\f {(x-x_0)^{\perp}}{2t_0}=0,\\
    &\int_M\left(\f{|x-x_0|^2}{4t_0^2}-\f n{2t_0}\right)e^{-\f{|x-x_0|^2}{4t_0}}d\mu=0\\
    &\int_M\f {x-x_0}{2t_0}e^{-\f{|x-x_0|^2}{4t_0}}d\mu=0.
\end{align*}
Then from the equations \eqref{j-1} - \eqref{I-2}, we have the second variation of $\mathcal{F}_{x_0,t_0}$ at the critical point $M$ as following:
\begin{align}
\mathcal{F}''&=\int_M(I''J+2I'J'+IJ'')d\mu\nonumber\\
&=(4\pi t_0)^{-\f n2}\int_M\left[|\nabla^{\perp}V|^2-\sum_{\alpha,\beta}\sigma_{\alpha\beta}V^{\alpha}V^{\beta}+\textrm{div } (\bar{\nabla}_VV)^T-<\bar{\nabla}_VV,H>\right.\nonumber\\
&\qquad+ h\left(\f{|x-x_0|^2}{4t_0^2}-\f n{2t_0}\right)'-<\left(\f {x-x_0}{2t_0}\right)',V-y>-<\f {x-x_0}{2t_0},\bar{\nabla}_VV>\label{vv-1}\\
&\qquad\left.+\left(h\left(\f{|x-x_0|^2}{4t_0^2}-\f n{2t_0}\right)+<\f {x-x_0}{2t_0},y>\right)^2\right]e^{-\f{|x-x_0|^2}{4t_0}}d\mu.\nonumber
\end{align}
The third and fourth integrals on the right hand side of \eqref{vv-1} can be rewritten as
\begin{align*}
    &~(4\pi t_0)^{-\f n2}\int_M\left(\textrm{div } (\bar{\nabla}_VV)^T-<\bar{\nabla}_VV,H>\right)e^{-\f{|x-x_0|^2}{4t_0}}d\mu\\
    =&~(4\pi t_0)^{-\f n2}\int_M\left(\sum_{i}\nabla_{e_i}<\bar{\nabla}_VV,e_i>+<\bar{\nabla}_VV,\f {(x-x_0)^{\perp}}{2t_0}>\right)e^{-\f{|x-x_0|^2}{4t_0}}d\mu\\
    =&~(4\pi t_0)^{-\f n2}\int_M\left(<\bar{\nabla}_VV,\f {(x-x_0)^T}{2t_0}>+<\bar{\nabla}_VV,\f {(x-x_0)^{\perp}}{2t_0}>\right)e^{-\f{|x-x_0|^2}{4t_0}}d\mu\\
    =&~(4\pi t_0)^{-\f n2}\int_M<\bar{\nabla}_VV,\f {x-x_0}{2t_0}>d\mu,
\end{align*}
and can be canceled with the seventh integral of \eqref{vv-1}. The fifth term is given by
\begin{align}
    \left(\f{|x-x_0|^2}{4t_0^2}-\f n{2t_0}\right)'&=\f {<x-x_0,V-y>}{2t_0^2}-h\left(\f {|x-x_0|^2-nt_0}{2t_0^3}\right).
\end{align}
For the sixth term, we have
\begin{align}
    \left(\f {x-x_0}{2t_0}\right)'=\f {V-y}{2t_0}-h\f {x-x_0}{2t_0^2}
\end{align}
Then for $x_0=0$ and $t_0=\f 12$, the second variation formula \eqref{vv-1} becomes
\begin{align*}
    \mathcal{F}''&=(2\pi)^{-\f n2}\int_M \biggl( |\nabla^{\perp}V|^2-\sum_{\alpha,\beta}\sigma_{\alpha\beta}V^{\alpha}V^{\beta}\\
    &\qquad\quad +4h<x,V-y>-4h^2(|x|^2-\f n2)-|V-y|^2\\
    &\qquad\quad+h^2(|x|^2-n)^2+<x,y>^2+2h\left(|x|^2-n\right)<x,y>\biggr) e^{-\f{|x|^2}2}d\mu.
\end{align*}
Using the equations \eqref{v-2} - \eqref{v-6}, and Corollary \ref{coor-v1}, we get
 \begin{align*}
     \mathcal{F}''&=(2\pi)^{-\f n2}\int_M  \biggl[ |\nabla^{\perp}V|^2-\sum_{\alpha,\beta}\sigma_{\alpha\beta}V^{\alpha}V^{\beta}-|V|^2\\
    &\qquad -2h^2|H|^2-4h<H,V>-2<y,V>-|y^{\perp}|^2\biggr]e^{-\f{|x|^2}2}d\mu.
\end{align*}
\end{proof}

When $p=1$, that is for hypersurface case, we have the following immediate corollary,
\begin{corr}[Theorem 4.14 in \cite{CM2}]
Suppose $M$ is a critical point of the functional $\mathcal{F}_{x_0,t_0}$, and $M$ is complete, with polynomial volume growth. If $M_s\subset\R^{n+1}$ is a normal variation of $M$, $x_s$,$t_s$ are variations of $x_0$ and $t_0$ with
 \begin{equation*}
    M_0'=fe_{n+1},\qquad x_0'=y,\qquad t_0'=h,
 \end{equation*}
then for $x_0=0$ and $t_0=\f 12$, we have the second variation formula
\begin{align}
    \mathcal{F}''=&(2\pi)^{-\f n2}\int_M\biggl[-fLf-2h^2|H|^2-4fh|H|\nonumber\\
    &\qquad -2f<y,e_{n+1}>-|y^{\perp}|^2\biggr]e^{-\f{|x|^2}2}d\mu,\label{stab-2}
\end{align}
where $L$ is the stability operator defined as
\begin{equation}
    Lf=\Delta f-<x,\nabla f>+|A|^2f+f.
\end{equation}
\end{corr}

\section{$\mathcal{F}$-stability and eigenvalues of $\mathcal{I}$}

In this section, we give the definition of $\mathcal{F}$-stability, then consider two eigenvector fields corresponding to the bilinear symmetric form $\mathcal{I}$ on the cross-sections in $NM$ and give a characterization of $\mathcal{F}$-stability in terms of the eigenvalues of $\mathcal{I}$.
\begin{defn}\label{defn3-1}
A critical point $M$ for $\mathcal{F}_{x_0,t_0}$ is $\mathcal{F}$-stable if for every compactly supported normal variation $V$ of $M$, there exist variations $x_s$ of $x_0$ and $t_s$ of $t_0$ that make $\mathcal{F}''\geq 0$ at $s=0$.
\end{defn}
In the remaining of this paper, without loss of generality, we only consider the case $x_0=0, t_0=\f 12$. The self-shrinker equation \eqref{shrinker} is equivalent to
\begin{equation}
    H=-x,
\end{equation}
or in terms of the components
\begin{equation}
    H^{\alpha}=-<x,e_{\alpha}>, \quad n+1\leq \alpha\leq n+p.
\end{equation}
The first and second covariant derivatives of $H$ have the following components (cf. \cite{CL2,LW}):
\begin{align}
    &H^{\alpha}_{,i}=\sum_jh^{\alpha}_{ij}<x,e_j>,\label{eq2-2}\\
    &H^{\alpha}_{,ij}=\sum_{k}h^{\alpha}_{ijk}<x,e_k>+h^{\alpha}_{ij}-\sum_{\beta}\sigma_{\alpha\beta}h^{\beta}_{ij},\\
    &\Delta^{\perp}H^{\alpha}=\sum_k<x,e_k>H^{\alpha}_{,k}+H^{\alpha}-\sum_{\beta}\sigma_{\alpha\beta}H^{\beta}.\label{eq2-4}
\end{align}

Let $x:M^n\ra\R^{n+p}$ be a closed self-shrinker, $V,W\in NM$ be two arbitrary normal variation vector fields. We set
\begin{align*}
\mathcal{I}(V,W)=&\int_M\biggl(<\nabla^{\perp}V,\nabla^{\perp}W>^2-\sum_{\alpha\beta}\sigma_{\alpha\beta}V^{\alpha}W^{\beta}-<V,W>\biggr)e^{-\f 12|x|^2}d\mu\\
=&-\int_M\sum_{\alpha}V^{\alpha}LW^{\alpha}e^{-\f 12|x|^2}d\mu.
\end{align*}
From the standard facts about the elliptic differential operator, we see that $\mathcal{I}$ is a symmetric bilinear form on the space of cross-sections in $NM$, which may be diagonalized with respect to the weighted $L^2$ inner product
\begin{align*}
    <V,W>_w=&\int_M<V,W>e^{-\f 12|x|^2}d\mu,
\end{align*}
and has distinct real eigenvalues $\{\mu_i\}$ such that
\begin{align*}
    \mu_1<\mu_2\leq \mu_3\leq \cdots \ra+\infty.
\end{align*}
Moreover, the dimension of each eigenspace is finite. In the sequel, we denote $\mathcal{W}_{\mu_i}$ the eigenspace corresponding to the eigenvalue $\mu_i$.
\begin{prop}\label{prop-5-1}
Let $x:M^n\ra\R^{n+p}$ be a closed self-shrinker. Then the mean curvature vector $H=\sum\limits_{\alpha}H^{\alpha}e_{\alpha}$ and the normal part $V^{\perp}=\sum\limits_{\alpha}<V,e_{\alpha}>e_{\alpha}$ of a constant vector field $V$  satisfies
\begin{equation}
    LH^{\alpha}=2H^{\alpha},\quad LV^{\alpha}=V^{\alpha},\quad n+1\leq \alpha\leq n+p,
\end{equation}
where the operator $L$ defined in \eqref{ope-vv}.  Therefore $H$ and $V^{\perp}$ are eigenvector fields of $\mathcal{I}$, and $H\in \mathcal{W}_{-2}$, $V^{\perp}\in \mathcal{W}_{-1}$.
\end{prop}

\begin{proof}
By the definition of $L$, the equation \eqref{eq2-4} implies
\begin{equation}
    LH^{\alpha}=2H^{\alpha},\qquad n+1\leq \alpha\leq n+p.
\end{equation}
For a constant vector field $V\in\R^{n+p}$, set $V^{\alpha}=<V,e_{\alpha}>$. We have
\begin{align*}
    V^{\alpha}_{,i}=&-h^{\alpha}_{ij}<V,e_j>,\\
    \Delta^{\perp}V^{\alpha}=&-H^{\alpha}_{,j}<V,e_j>-\sum_{\alpha\beta}\sigma_{\alpha\beta}<V,e_{\beta}>\\
    =&-\sum_k<x,e_k>h^{\alpha}_{jk}<V,e_j>-\sum_{\alpha\beta}\sigma_{\alpha\beta}V^{\beta}\\
    =&\sum_k<x,e_k>V^{\alpha}_{,k}-\sum_{\alpha\beta}\sigma_{\alpha\beta}V^{\beta}.
\end{align*}
It follows that
\begin{equation}
    LV^{\alpha}=V^{\alpha}.
\end{equation}
Therefore
\begin{align*}
    \mathcal{I}(H,H)=&-2<H,H>_w,\\
    \mathcal{I}(V^{\perp},V^{\perp})=&-<V^{\perp},V^{\perp}>_w.
\end{align*}
i.e., $H\in \mathcal{W}_{-2}$, $V^{\perp}\in \mathcal{W}_{-1}$.
\end{proof}

Now we derive the following necessary condition for closed $\mathcal{F}$-stable self-shrinkers:

\begin{prop}
Suppose $x:M^n\ra\R^{n+p}$ is a closed $\mathcal{F}$-stable self-shrinker, then $\{-2,-1\}$ are the only negative eigenvalues of the bilinear symmetric form $\mathcal{I}$.
\end{prop}
\begin{proof}
We prove the theorem by a contradiction. Suppose there exists another eigenvector field $V\in NM$ of $\mathcal{I}$ corresponding to the eigenvalue $\mu<0$ and $\mu\neq -2,-1$. Since the eigenvector fields corresponding to different eigenvalues are orthogonal with respect to the weighted $L^2$ inner product, we have
\begin{align*}
    &\int_M<H,V>e^{-\f 12|x|^2}d\mu=0,\quad \int_M<y^{\perp},V>e^{-\f 12|x|^2}d\mu=0,
\end{align*}
for any constant vector field $y\in\R^{n+p}$. Then put $V$ into the second variation formula \eqref{vv-3'}, we have
\begin{align*}
    \mathcal{F}''=&(2\pi)^{-\f n2}\int_M\biggl[-\sum_{\alpha}V^{\alpha}LV^{\alpha}-2h^2|H|^2\\
    &\qquad -4h<H,V>-2<y,V>-|y^{\perp}|^2\biggr]e^{-\f 12|x|^2}d\mu\\
    =&(2\pi)^{-\f n2}\int_M\biggl[-\sum_{\alpha}V^{\alpha}LV^{\alpha}-2h^2|H|^2-|y^{\perp}|^2\biggr]e^{-\f 12|x|^2}d\mu\\
    \leq &(2\pi)^{-\f n2}\mathcal{I}(V,V)\\
    <&0
\end{align*}
for any choice of $h\in\R$ and $y\in\R^{n+p}$. This implies $M^n$ is $\mathcal{F}$-unstable.
\end{proof}

\section{$\mathcal{F}$-stability for closed self-shrinkers}

In higher codimension, the study of $\mathcal{F}$-stability of self-shrinkers becomes complicated as the codimension increases. First we will prove the n-sphere $\SS^n(\sqrt{n})$ in $\R^{n+p}$ is $\mathcal{F}$-stable.

\begin{prop}
The n-sphere $x:\SS^n(\sqrt{n})\ra\R^{n+p}$ is $\mathcal{F}$-stable as a self-shrinker.
\end{prop}
\begin{proof}
Choose orthonormal basis $\{e_{n+1},\cdots,e_{n+p-1},e_{n+p}\}$ of the normal bundle of $\SS^n(\sqrt{n})$ such that $e_{n+p}$ is parallel to the mean curvature vector $H$, and $e_{n+1},\cdots,e_{n+p-1}$ are constant vector fields in $\R^{n+p}$. $e_{n+p}$ is parallel in the normal bundle, i.e., $\nabla^{\perp}e_{n+p}=0$. The second fundamental form satisfies $|A|^2=\sum\limits_{i,j}(h^{n+p}_{ij})^2=1$, and $h^{\alpha}_{ij}=0$ for $\alpha\neq n+p$. The position vector $x$ and the mean curvature vector $H$ satisfy $H=-x$ and $|H|^2=|x|^2=n$.

For any variation vector field $V=\sum V^{\alpha}e_{\alpha}$. Since the position vector $x$ in normal on $\SS^n(\sqrt{n})$, we have $<x,e_k>$=0 for all $1\leq k\leq n$. It follows that the second variation formula \eqref{vv-3} becomes
\begin{align*}
    \mathcal{F}''=&(2\pi)^{-\f n2}e^{-\f n2}\int_M\biggl[-V^{n+p}L_{\nu}V^{n+p}-2nh^2-4\sqrt{n}hV^{n+p}\\
    &\quad -2V^{n+p}<y,e_{n+p}>-<y,e_{n+p}>^2+\sum_{\alpha\neq n+p}\biggl(|\nabla V^{\alpha}|^2-|V^{\alpha}|^2\\
    &\quad -2<y,V^{\alpha}e_{\alpha}>-<y,e_{\alpha}>^2\biggr)\biggr]d\mu,
\end{align*}
where $\nabla$ denotes the gradient of functions on $\SS^n(\sqrt{n})$,  the operator $L_{\nu}=\Delta +2$ also acts on smooth functions on $\SS^n(\sqrt{n})$.

Recall that the eigenvalues of $\Delta$ on the sphere $\SS^n(\sqrt{n})$ are given by (see \cite{Cha})
\begin{align}
    \mu_k=\f{k^2+(n-1)k}{n}.
\end{align}
Clearly, the constant functions are eigenfunctions corresponding to the zero eigenvalue $\mu_0=0$. Note that the position vector $x$ satisfies $\Delta x=H=-x$, so for any constant vector $z\in\R^{n+p}$, we have
\begin{align*}
    -\Delta <z,e_{n+p}>=\Delta <z,\f{x}{\sqrt{n}}>=<z,e_{n+p}>,
\end{align*}
i.e., $<z,e_{n+p}>$ are eigenfunction of $\Delta$ corresponding to the first eigenvalue $\mu_1=1$. Now we choose $a\in\R$ a constant real number and $z\in\R^{n+p}$ a constant vector such that
\begin{align}
    V^{n+p}=f_0+a+<z,e_{n+p}>,
\end{align}
with $f_0$ in the space spanned by all the eigenfunctions for $\mu_k (k\geq 2)$ of $\Delta$ on $\SS^n(\sqrt{n})$.  By the orthogonality of the different eigenspaces, we have
\begin{align*}
&\int_M\biggl(-V^{n+p}L_{\nu}V^{n+p}-2nh^2-4\sqrt{n}hV^{n+p}\\
&\qquad-2V^{n+p}<y,e_{n+p}>-<y,e_{n+p}>^2\biggr)d\mu\\
=&\int_M\biggl(-f_0(\Delta f_0+2f_0)-2(a+\sqrt{n}h)^2-<z+y,e_{n+p}>^2\biggr)d\mu\\
\geq & \int_M\biggl(\f 2nf_0^2-2(a+\sqrt{n}h)^2-<z+y,e_{n+p}>^2\biggr)d\mu,
\end{align*}
which can be made nonnegative by choosing $h=-a/{\sqrt{n}}$ and
\begin{align}
    y-\sum_{\alpha=n+1}^{n+p-1}<y,e_{\alpha}>e_{\alpha}=z-\sum_{\alpha=n+1}^{n+p-1}<z,e_{\alpha}>e_{\alpha}.
\end{align}
On the other hand, since $e_{\alpha}, \alpha=n+1,\cdots,n+p-1$ are constant vectors in $\R^{n+p}$, if for some $\alpha\neq n+p$, $V^{\alpha}$ is constant function, then we can choose $<y,e_{\alpha}>=-V^{\alpha}$ such that
\begin{align*}
    &\int_M\biggl(|\nabla V^{\alpha}|^2-|V^{\alpha}|^2-2<y,V^{\alpha}e_{\alpha}>-<y,e_{\alpha}>^2\biggr)d\mu=0.
\end{align*}
If $V^{\alpha}$ is not a constant, since the first eigenvalue of $\Delta$ on sphere $\SS^n(\sqrt{n})$ is $\mu_1=1$, by choosing $<y,e_{\alpha}>=0$, we can obtain
\begin{align*}
    &\int_M\biggl(|\nabla V^{\alpha}|^2-|V^{\alpha}|^2-2<y,V^{\alpha}e_{\alpha}>-<y,e_{\alpha}>^2\biggr)d\mu\\
    =&\int_M\biggl(|\nabla V^{\alpha}|^2-|V^{\alpha}|^2\biggr)d\mu\\
    \geq &0.
\end{align*}
The vector $y$ we chosen as above is a constant vector in $\R^{n+p}$,  because all of $z,e_{n+1},\cdots,e_{n+p-1}$ are constant vectors in $\R^{n+p}$.
It follows that for any variation vector field $V$, we can choose constant real number $h\in\R$ and constant vector $y\in\R^{n+p}$ such that $\mathcal{F}''\geq 0$, this means that the sphere $\SS^n(\sqrt{n})$ is $\mathcal{F}$-stable.

\end{proof}

Conversely, we want to determine which closed self-shrinkers are $\mathcal{F}$-stable. In general, this is complicated in the higher codimension. In the following, we will consider the special situation, ``self-shrinkers with parallel principal normal''.
\begin{thm}
Let $x:M^n\ra\R^{n+p}$ be a closed $\mathcal{F}$-stable self-shrinker with parallel principal normal. Then $M^n$ is a minimal submanifold in the sphere $\SS^{n+p-1}(\sqrt{n})$.
\end{thm}
\begin{proof}
If $|H|\neq 0$ on $M$, then by Smoczyk's classification theorem \cite{Smo} on closed self-shrinker with parallel principal normal, $M^n$ is a minimal submanifold in the sphere $\SS^{n+p-1}(\sqrt{n})$.

If $H$ vanishes somewhere on $M$, we will show that $M$ is $\mathcal{F}$-unstable, and therefore contradicts with the assumption. The $\mathcal{F}$-unstable means that there exists a variation $V\in NM$ such that for any variation $y$ of $x_0=0$ and $h$ of $t_0=\f 12$, we always have $\mathcal{F}''< 0$. To prove $M$ is $\mathcal{F}$-unstable, we choose the variation vector field $V=fe_{n+p}$ with $f$ a smooth function on $M$ with $e_{n+p}$ parallel to the mean curvature vector $H$. Since $\nabla^{\perp}e_{n+p}=0$, for this vector $V=fe_{n+p}$, the second variation formula \eqref{vv-3} becomes
\begin{align}
    \mathcal{F}''=&(2\pi)^{-\f n2}\int_M\biggl[-fL_{\nu}f-2h^2|H|^2-4hf<H,e_{n+p}>\nonumber\\
    &\qquad -2f<y,e_{n+p}>-|y^{\perp}|^2\biggr]e^{-\f 12|x|^2}d\mu,\label{vv-parall}
\end{align}
where the operator $L_{\nu}$ is defined as
\begin{equation*}
   L_{\nu}f=\Delta f-<x,\nabla f>+|Z|^2 f+f,
\end{equation*}
with $|Z|^2=\sum\limits_{i,j}(h^{n+p}_{ij})^2$.

Since $e_{n+p}$ is parallel to $H$, we can write $H$ by
\begin{align*}
    H=\sum_{\alpha}H^{\alpha}e_{\alpha}=<H,e_{n+p}>e_{n+p},
\end{align*}
i.e., $H^{\alpha}=0$ for $\alpha\neq n+p$ and $H^{n+p}=<H,e_{n+p}>$. Note that $H^{\alpha}$ are the components of the tensor field $H$, and $<H,e_{n+p}>$ is just a function on $M$. Recall that for submanifold with parallel principal normal, we have (cf. \cite{LW})
\begin{align*}
    &H^{\alpha}_{,i}=0,\qquad H^{\alpha}_{,ij}=0, \alpha\neq n+p,\\
    &H^{n+p}_{,i}=<H,e_{n+p}>_{,i},\quad H^{n+p}_{,ij}=<H,e_{n+p}>_{,ij}.
\end{align*}
Combing with equation \eqref{eq2-4} gives
\begin{align}
    &L_{\nu}<H,e_{n+p}>=2<H,e_{n+p}>,\label{para-1}\\
    &L_{\nu}<y,e_{n+p}>=<y,e_{n+p}>, \quad y\in\R^{n+p}.\label{para-2}
\end{align}
The elliptic differential operator $L_{\nu}$ is self-adjoint with respect to the weighted $L^2$ inner product, then standard spectrum theory gives that $L_{\nu}$ has real eigenvalues
\begin{align*}
    \mu_1<\mu_2\leq\cdots\ra+\infty,
\end{align*}
and there are orthonormal basis $\{u_k\}$ for the weighted $L^2$ space with $L_{\nu}u_k=-\mu_ku_k$. The eigenfunctions corresponding to different eigenvalues are orthogonal with respect to the weighted $L^2$ inner product. Any eigenfunction corresponding to the smallest eigenvalue $\mu_1$ does not change sign. Therefore \eqref{para-1} and \eqref{para-2} imply $<H,e_{n+p}>$, $<y,e_{n+p}>$ are eigenfunction of $L_{\nu}$ corresponding to eigenvalues $-2$ and $-1$ respectively.

 Since $<H,e_{n+p}>$ vanishes somewhere on $M$, then $-2$ is not the smallest eigenvalue of the elliptic operator $L_{\nu}$. Thus there is a positive function $f$ with $-L_{\nu}f=\mu f$ $(\mu<-2)$. Then $f$ is orthogonal to $<H,e_{n+p}>$ and $<y,e_{n+p}>$ for $y\in\R^{n+p}$, i.e.,
 \begin{align*}
    &\int_Mf<H,e_{n+p}>e^{-\f 12|x|^2}d\mu=0,\\
    & \int_Mf<y,e_{n+p}>e^{-\f 12|x|^2}d\mu=0.
 \end{align*}
 Substituting these into \eqref{vv-parall} gives
 \begin{align}
 \mathcal{F}''=&(2\pi)^{-\f n2}\int_M\biggl(\mu f^2-2h^2|H|^2-|y^{\perp}|^2\biggr)e^{-\f 12|x|^2}d\mu\\
 < &0\nonumber
 \end{align}
 for any choice of $h$ and $y$. This means $M$ is $\mathcal{F}$-unstable, contradicts with the hypothesis of the theorem.
\end{proof}

\begin{rem}
From the proof of Theorem 6.2, we can see that for closed self-shrinker with parallel principal normal, if there is another negative  eigenvalue $\mu\neq -1, -2$ of $L_{\nu}$, i.e.,
\begin{align*}
    -L_{\nu}f=\mu f,\quad \mu\neq -1,-2,\mu<0
\end{align*}
for some eigenfunction $f$, then $M$ is $\mathcal{F}$-unstable.
\end{rem}

In general, not all the minimal submanifolds of spheres are $\mathcal{F}$-stable self-shrinkers, in the following, we will show that the only $\mathcal{F}$-stable one is the sphere $\SS^n(\sqrt{n})$.  The key observation is that the $\mathcal{F}$-stability is closely related to stability as a minimal surface of the sphere, and our argument is related to the argument of Simons \cite[Theorem 5.1.1]{Simon} on instability of minimal submanifolds of spheres.

\begin{thm}\label{thm-6-3}
Let $M^n$ be a closed minimal submanifold of $\SS^{n+p-1}(\sqrt{n})\subset\R^{n+p}$. If $M$ is $\mathcal{F}$-stable, then $M$ is the $n$-sphere $\SS^n(\sqrt{n})$.
\end{thm}

\begin{proof}
Let $x:\ M^n\ra\SS^{n+p-1}(\sqrt{n})$ be a closed minimal submanifold, then $H=-x$ and $|H|^2=|x|^2=n$.
Then at each point we have the orthogonal decomposition $\R^{n+p} = \R x\oplus T_xM\oplus N_xM$, where $N_xM$ is the normal bundle as a submanifold of $\SS^{n+p-1}$.

We choose the variation vector field
\begin{align}
    V(x)=\pi_{N_xM}(z),
\end{align}
where $z\in\R^{n+p}$ is a constant vector and $\pi_{N_xM}$ is the orthogonal projection. From  the computation in Proposition \ref{prop-5-1}, the second variation formula \eqref{vv-3'} becomes (in a local orthonormal frame $e_{\alpha}$ where $e_1,\ldots,e_n$ span $T_xM$, $e_{n+1},\ldots,e_{n+p-1}$ span $N_xM$, and $e_{n+p}$ is proportional to $x$)
\begin{align*}
    \mathcal{F}''=&(2\pi)^{-\f n2}e^{-\f n2}\int_M\biggl(-\sum_{\alpha\neq n+p}<z,e_{\alpha}>^2-2nh^2\\
    &\qquad -2\sum_{\alpha\neq n+p}<y,e_{\alpha}><z,e_{\alpha}>-|y^{\perp}|^2\biggr)d\mu\\
    =&-(2\pi)^{-\f n2}e^{-\f n2}\int_M\biggl(\sum_{\alpha\neq n+p}<z+y,e_{\alpha}>^2+2nh^2+<y,e_{n+p}>^2\biggr)d\mu\\
    =&-(2\pi)^{-\f n2}e^{-\f n2}\int_M\biggl(\left|\pi_{N_xM}(z+y)\right|^2+2nh^2+\left(y\cdot \f{x}{|x|}\right)^2
    \biggr)\,d\mu.
\end{align*}
If $M$ is $\mathcal{F}$-stable, then for any such $z$ there must exist some $h\in\R$ and constant vector $y\in\R^{n+p}$ such that $\mathcal{F}''\geq 0$. It follows from the above variation formula that $h=0$, and we necessarily have
\begin{equation}\label{eq:stab-conditions}
    <y,x>= 0,\quad\text{and}\quad z+y\in \R x\oplus T_xM,
\end{equation}
for every $x\in M$.

Let $V$ be the subspace of $\R^{n+p}$ defined by
\begin{align*}
    V=\{y\in\R^{n+p}:<y,x>= 0, \textrm{for all } x\in M\}.
\end{align*}
For $y\in V$ we also have $y\cdot v=0$ for any $v\in T_xM$, by differentiating the equation $y\cdot x=0$.  Therefore $V$ is orthogonal to $\R x\oplus T_xM$, and so $V$ is a subspace of $N_xM$.  In particular $V$ has dimension at most $p-1$.  Furthermore, if $V$ has dimension $p-1$ then we have $M\subset \SS^{n+p-1}\cap V^{\perp}$, which is an $n$-dimensional sphere.  It then follows by connectedness that $M$ is itself a totally geodesic $n$-dimensional sphere in $\SS^{n+p-1}$.

Now fix $x\in M$.  For any $z\in \R^{n+p}$, equation \eqref{eq:stab-conditions} implies that there exists $y\in V$ with $z+y\in \R x\oplus T_xM$, so that $z\in V\oplus \R x\oplus T_xM$.  Since $z$ is arbitrary, we have $\R^{n+p}= V\oplus \R x\oplus T_xM$, from which it follows that $V$ has dimension at least $p-1$.
Therefore $V$ has dimension exactly $p-1$ and $M$ is the $n$-sphere $\SS^{n+p-1}\cap V^{\perp}$. This completes the proof.
\end{proof}

Theorem \ref{thm-1-2} is a direct consequence of Proposition 6.1 and Theorems 6.2, 6.3.

\section{$\mathcal{F}$-stability for complete noncompact self-shrinker}

In this section, we suppose $x:M^n\ra\R^{n+p}$ is a complete noncompact self-shrinker with parallel principal normal and polynomial volume growth. We will show that the only $\mathcal{F}$-stable one is the plane $\R^n$. First we have the following two lemmas.

\begin{lem}\label{lem-7-1}
Let $N^k$ be a closed minimal submanifold in $\SS^{k+p-1}(\sqrt{k})$, then $x:M^n=N^k\times\R^{n-k}\ra\R^{n+p}$ is $\mathcal{F}$-unstable as a self-shrinker.
\end{lem}
\begin{proof}
We choose local orthonormal frame $\{e_{\alpha}\}$ for the normal bundle of $M$ such that $e_{n+p}$ is proportional to the mean curvature vector $H$. We set the variation vector $V=fe_{n+p}$, we want to find some function $f$ with compact support such that the second variation $\mathcal{F}''$ is negative for every choice of $h$ and $y$. Since $\nabla^{\perp}e_{n+p}=0$, $|H|^2=|x^{\perp}|^2=k$ and $|Z|^2=\sum\limits_{i,j}(h^{n+p}_{ij})^2=1$, as derived in the previous section, we have the second variation formula
\begin{align}
    \mathcal{F}''=&(2\pi)^{-\f n2}\int_M\biggl[-fL_{\nu}f-2kh^2-4\sqrt{k}hf\nonumber\\
    &\qquad -2f<y,e_{n+p}>-|y^{\perp}|^2\biggr]e^{-\f 12|x|^2}d\mu,\label{eq-7-1}
\end{align}
where the operator $L_{\nu}$ is defined as
\begin{equation*}
   L_{\nu}f=\Delta f-<x,\nabla f>+2f.
\end{equation*}
Let $x_1$ be the coordinate function corresponding to the first coordinate in the $\R^{n-k}$, then Lemma \ref{lem-3-2} implies
\begin{align}\label{eq-7-2}
    &\mathcal{L}x_1=-x_1,  &L_{\nu}x_1=x_1.
\end{align}
Since $M$ has polynomial volume growth, it follows from the self-adjointness of $\mathcal{L}$ in the weighted $L^2$ space that
\begin{align}\label{eq-7-3}
    0=\int_M\mathcal{L}x_1e^{-\f 12|x|^2}d\mu=-\int_Mx_1e^{-\f 12|x|^2}d\mu.
\end{align}
For any constant vector $y\in\R^{n+p}$, we know that $<y,e_{n+p}>$ is an eigenfunction of $\Delta_{N^k}$ on the $N^k$ factor (cf. \cite{Ta}). Let $x'=(x_1,\cdots,x_{n-k})$ be the coordinates of $\R^{n-k}$, then $<y,e_{n+p}>$ is independent of $x'$. Moreover, it follows from the Fubini's theorem that for any bounded function $\phi(x')$ on $\R^{n-k}$,
\begin{align}\label{eq-7-4}
    \int_M\phi(x')x_1<y,e_{n+p}>e^{-\f 12|x|^2}d\mu=0.
\end{align}

Now we suppose $\phi_j(x')$ is a cutoff function on $\R^{n-k}$ which is equal to one on $B_j$, and zero outside $B_{j+1}$, where $B_j$ denotes the Euclidean ball in $\R^{n-k}$ with radius $j$.  Then we choose $f_j=\phi_j x_1$ which has compact support in $M$, and set the variation vector $V=f_je_{n+p}$, it follows from \eqref{eq-7-1} and \eqref{eq-7-4} that
\begin{align*}
    \mathcal{F}''=&(2\pi)^{-\f n2}\int_M\biggl[-f_jL_{\nu}f_j-2kh^2-4\sqrt{k}hf_j-|y^{\perp}|^2\biggr]e^{-\f 12|x|^2}d\mu.
\end{align*}
Let $j\ra\infty$, by using \eqref{eq-7-2} and \eqref{eq-7-3}, the dominated convergence theorem gives
\begin{align}
    &\lim_{j\ra\infty}\int_M-f_jL_{\nu}f_je^{-\f 12|x|^2}d\mu=-\int_M|x_1|^2e^{-\f 12|x|^2}d\mu\\
    &\lim_{j\ra\infty}\int_Mf_je^{-\f 12|x|^2}d\mu=0.
\end{align}
Therefore, for $j$ sufficiently large, we have
\begin{align}
    \mathcal{F}''\leq -\f 12(2\pi)^{-\f n2}\int_M|x_1|^2e^{-\f 12|x|^2}d\mu,
\end{align}
which is negative no matter what values of $y$ and $h$ which we choose.

\end{proof}

\begin{lem}\label{lem-7-2}
Let $x:M^n\ra\R^{n+p}$ be a complete noncompact self-shrinker with parallel principal normal. If $H$ vanishes somewhere but not identically, then $M$ is $\mathcal{F}$-unstable.
\end{lem}
\begin{proof}
We will follow the argument in Colding-Minicozzi II's paper \cite{CM2} closely to show that there exists a variation such that $\mathcal{F}''$ is negative for every choice of $h$ and $y$. Here we only give the outline, the reader can refer \cite{CM2} for the details of the argument.

For a complete noncompact self-shrinker $M^n$ in $\R^{n+p}$ with parallel principal normal, we have the second variation formula
 \begin{align*}
    \mathcal{F}''=&(2\pi)^{-\f n2}\int_M\biggl[-fL_{\nu}f-2h^2|H|^2-4hf<H,e_{n+p}>\\
    &\qquad -2f<y,e_{n+p}>-|y^{\perp}|^2\biggr]e^{-\f 12|x|^2}d\mu,
\end{align*}
with the operator $L_{\nu}$ defined by
\begin{align*}
    L_{\nu}f=\Delta f-<x,\nabla f>+|Z|^2f+f.
\end{align*}
Since $M$ is noncompact, there may not be the first eigenvalue for $L_{\nu}$. However, we can still define the bottom of the spectrum $\mu_1$ by
\begin{align}
\mu_1=\inf_f\f{\int_M\biggl(|\nabla f|^2-|Z|^2f^2-f^2\biggr)e^{-\f 12|x|^2}d\mu}{\int_Mf^2e^{-\f 12|x|^2}d\mu},
\end{align}
where the infimum is taken over all smooth functions with compact support. By using the standard density arguments and the dominated convergence theorem, we can show that we get the same $\mu_1$ by taking the infimum over all Lipschitz functions $f$ satisfying
\begin{align}\label{cond-7}
    \int_M\biggl(f^2+|\nabla f|^2+|Z|^2f^2\biggr)e^{-\f 12|x|^2}d\mu<\infty.
\end{align}

If $\mu_1\neq -\infty$, it can be proved that there is a positive function $u$ on $M$ with $L_{\nu}u=-\mu_1u$. And furthermore, if $v$ is in the weighted $W^{1,2}$ space (that is both $v$ and $\nabla v$ are in the weighted $L^2$ space) and $L_{\nu}v=-\mu_1v$, then $v=cu$ for some constant $c\in\R$.

Recall that $<H,e_{n+p}>$ satisfies
\begin{align*}
    L_{\nu}<H,e_{n+p}>=2<H,e_{n+p}>,
\end{align*}
and it can be easily checked that $<H,e_{n+p}>$ satisfies the condition \eqref{cond-7}, so we get $\mu_1\leq -2$. But $<H,e_{n+p}> $ vanishes somewhere by the assumption of our Lemma, therefore $\mu_1<-2$.

Then we can show that the lowest eigenfunction on a sufficiently large ball is almost orthogonal to $<H,e_{n+p}>$. By using this fact, we can construct a variation to get the instability. Precisely, there exists a $\bar{R}$ so that if $R>\bar{R}$ and $f$ is a Dirichlet eigenfunction for the first eigenvalue $\mu_1(B_R)$, then for any $h\in\R$ and any $y\in\R^{n+p}$ we have
\begin{align*}
    &\int_{M\cap B_R}\biggl[-fL_{\nu}f-2h^2|H|^2-4hf<H,e_{n+p}>\\
    &-2f<y,e_{n+p}>-|y^{\perp}|^2\biggr]e^{-\f 12|x|^2}d\mu<0.
\end{align*}
If we fix a $R\geq\bar{R}$, $f$ is a Dirichlet eigenfunction for the first eigenvalue $\mu_1(B_R)$ and set $V=fe_{n+p}$. Then for this variation vector $V$, $\mathcal{F}''$ is negative for every choice of $h$ and $y$ and this completes the proof.
\end{proof}

Now we are in the position to prove Theorem \ref{thm-1-3}.

\begin{proof}[Proof of Theorem \ref{thm-1-3}]
If the mean curvature vector $H$ does not vanish anywhere, then the assumption of our theorem and Theorem 1.1 in \cite{LW} (see also \cite[Theorem 1.3]{Smo}) give that $M^n=N^k\times\R^{n-k}$. It follows from lemma \ref{lem-7-1} that $M^n$ is $\mathcal{F}$-unstable.

If the mean curvature vector $H$ vanishes somewhere but not identically, then Lemma \ref{lem-7-2} also implies $M^n$ is $\mathcal{F}-$unstable.

So $H$ must vanish identically, and therefore $M^n$ is the plane $\R^n$.
\end{proof}

\section{Entropy stable self-shrinkers}

Finally, in the last section, we relate the entropy-stability to the $\mathcal{F}$-stability of self-shrinkers with higher codimension. We note that in the previous two sections, we have added the condition ``with parallel principal normal" on the self-shrinker. The condition is not preserved by the mean curvature flow and seems artificial, but it includes the important examples of minimal submanifolds in the sphere.

In \cite{CM2}, Colding and Minicozzi proved that entropy stable self-shrinkers that do not split off a line must be $\mathcal{F}$-stable. In higher codimension case, we also have the same result. This follows from the same argument of Colding-Minicozzi,  with some small changes of the notations in the first and second variation formulas of the $\mathcal{F}$-functionlal. We omit the details of the proof.

\begin{prop}\label{prop-8-1}
Suppose $x:M^n\ra\R^{n+p}$ is a smooth complete self-shrinker without boundary, with polynomial volume growth, and does not split off a line isometrically. If $M$ is $\mathcal{F}$-unstable, then there is a compactly supported variation $M_s$ of $M$ such that the entropy satisfies $\lambda(M_s)<\lambda(M)$ for $s\neq 0$.
\end{prop}

Combining Theorem \ref{thm-1-2}, \ref{thm-1-3} and Proposition \ref{prop-8-1} gives the following classifications of entropy stable self-shrinkers.

\begin{corr}\label{thm-8-2}
Suppose that $M^n$ is an n-dimensional closed self-shrinker with parallel principal normal in $\R^{n+p}$, but not the n-sphere $\SS^n(\sqrt{n})$. Then $M$ can be perturbed to an arbitrarily close submanifold $\td{M}^n\subset\R^{n+p}$ such that $\lambda(\td{M})<\lambda(M)$.
\end{corr}
\begin{proof}
Since $M^n$ is a closed self-shrinker with parallel principal normal in $\R^{n+p}$, but not the n-sphere $\SS^n(\sqrt{n})$. Theorem \ref{thm-1-2} implies $M$ is $\mathcal{F}$-unstable. On the other hand, $M$ clearly does not split off a line, so Proposition \ref{prop-8-1} gives that it is entropy unstable.
\end{proof}

\begin{corr}
Suppose $M^n$ is a complete noncompact self-shrinker in $\R^{n+p}$ with parallel principal normal, with polynomial volume growth and without boundary. If $|A|^2-|A^{\nu}|^2\leq c$ for some constant $c$ on $M$ and $M$ is not equal to $\SS^k(\sqrt{k})\times\R^{n-k}$, then $M$ can be perturbed to an arbitrarily close submanifold $\td{M}^n\subset\R^{n+p}$ such that $\lambda(\td{M})<\lambda(M)$.
\end{corr}
\begin{proof}
To prove it, suppose that $M^n=N^k\times\R^{n-k}\subset \R^{n+p}$, where $N^k$ is a self-shrinker in $\R^{k+p}$ with parallel principal normal and does not split off another line isometrically. By the assumption, $N^k$ is not the sphere $\SS^k(\sqrt{k})$. From Theorem \ref{thm-1-3}, Proposition \ref{prop-8-1} and Corollary \ref{thm-8-2} we conclude that $N^k$ can be perturbed to an arbitrarily close $\td{N}^k$  such that $\lambda_{\R^{k+p}}(\td{N})<\lambda_{\R^{k+p}}(N)$. Note that by a direct calculation, for $M^n=N^k\times\R^{n-k}$ where $N^k\subset \R^{k+p}$ we have the following fact (see \cite{CM2}):
\begin{align*}
    \mathcal{F}_{x,t_0}^{\R^{n+p}}(M)=\mathcal{F}_{x',t_0}^{\R^{k+p}}(N),
\end{align*}
where $x'$ is the projection of $x$ to $\R^{k+p}$. Then the Corollary follows easily from the definition of entropy.
\end{proof}

\bibliographystyle{Plain}

\end{document}